\documentclass{amsart}
\usepackage{amssymb, amsmath,amsthm}
\usepackage{amsaddr}
\usepackage{graphicx}
\usepackage{mathrsfs}
\usepackage{color}
\usepackage{esint}
\usepackage{enumerate}
\usepackage[colorlinks=true, linkcolor=blue,citecolor=blue]{hyperref}
\newtheorem{thm}{Theorem}[section]
\newtheorem{cor}[thm]{Corollary}
\newtheorem{lem}[thm]{Lemma}

\theoremstyle{definition}

\newtheorem{rem}[thm]{Remark}

\newcommand{\R}{\mathbb R}

\newcommand{\N}{\mathbb{N}}
\newcommand{\K}{\mathbb{K}}

\DeclareMathOperator{\rank}{rank}
\title{The Kalman condition for the boundary controllability of coupled 1-d wave equations.}
\email{saavdonin@alaska.edu, ldeteresa@im.unam.mx}
\begin{document}
\renewcommand\abstractname{\textbf{Abstract}}
\maketitle

\begin{center}
S. Avdonin\textsuperscript{a,}, L. de Teresa\textsuperscript{b,}\\
\vspace{0.1in}
\begin{footnotesize}
\emph{\textsuperscript{a}University of Alaska Fairbanks}\\
\emph{\textsuperscript{b}Instituto de Matem\'aticas, Universidad Nacional Aut\'onoma de M\'exico, Circuito Exterior, G. U. 04510 D.F., M\'exico}
\end{footnotesize}
\end{center}

\noindent \rule{\textwidth}{0.4pt}
\begin{abstract} \mbox{} \\
This paper is devoted to prove the exact controllability of a system of $N$ one-dimensional coupled wave equations when the control is exerted on a part of the boundary by means of one control. We consider the case where the coupling matrix $A$ has distinct eigenvalues. We give a \emph{Kalman condition} (necessary and sufficient) and give a description, non-optimal in general, of the attainable set. \\

\end{abstract}
\noindent \rule{\textwidth}{0.4pt}

\section{Statement of the Problem and Main Results}
This work is devoted to the study of the controllability properties of the following hyperbolic system
\begin{equation}
\left\{ \begin{array}{ll} u_{tt} - u_{xx} + Au = 0, & \text{in $Q = (0,\pi) \times (0,T)$,}\\
u(0,t) = bf(t), \quad u(\pi,t) = 0 & \text{for $t \in (0,T)$,}\\
u(x,0) = u^0(x), \quad u_t (x,0) = u^1(x) & \text{for $x \in (0,\pi)$,} \end{array} \right.
\end{equation}
where $T > 0$ is given, $A \in \mathcal{L} (\R^N)$ is a given matrix, $b$ a given vector from $\R^N$ and $f \in L^2 (0,T)$ is a control function to be determined which acts on the system by means of the Dirichlet boundary condition at the point $x = 0$. The initial data $(u^0, u^1)$ will belong to a Hilbert space $\mathcal{H}$, which is to be specified in our main result. Our goal is to give necessary and sufficient conditions for the exact controllability of System (1) and the space $\mathcal{H}$ where this can be done. 

We recall that System (1) is exactly controllable in $\mathcal{H}$ at time $T$ if, for every initial and final data $(u^0,u^1), (z^0,z^1)$, both in $\mathcal{H}$, there exists a control $f \in L^2 (0,T)$ such that the solution of System (1) corresponding to $(u^0,u^1,f)$ satisfies
\begin{equation}
u(x,T) = z^0 (x), \quad u_t(x,T) = z^1 (x). 
\end{equation}

Due to the linearity and time reversibility of System (1), this is equivalent to exact controllability from zero at time $T$. In other words, System (1) is exactly controllable if for every final state $(z^0,z^1) \in \mathcal{H}$, there exists a control $f \in L^2 (0,T)$ such that the solution $u$ to System (1) corresponding to $f$ satisfies (2) and
\begin{equation}
u(x,0) = 0 = u_t (x,0).
\end{equation}
For this reason, we will assume that $u^0 \equiv 0, u^1 \equiv 0$. 

As of now, the controllability properties of System (1) are well known in the scalar case, i.e. when $N = 1$ (see for example \cite{fattorini1977}). When $N = 1$ and $b \not\equiv 0$, System (1) is exactly controllable in $\mathcal{H} = L^2 (0,\pi) \times H^{-1} (0,\pi)$ if $T \geq T_0 = 2\pi$. 

Most of the known controllability results of (1) are in the case of two coupled equations: see \cite{acd2013, rd2011, bat2017}, 
but the results are for a particular coupling matrix $A$. In the $d$-dimensional situation, that is, for a system of coupled wave equations in a domain $\Omega \subset \R^d$, Alabau-Boussouria and collaborators have obtained several results in the case of two equations for particular coupling matrices (see e.g. \cite{alabau2003, alabau2014, al2011} and the references therein). 

On the other hand, controllability properties of linear ordinary differential systems are well understood. In particular, we have the famous Kalman rank condition (see for example \cite{kfa1969} Chapter 2, p.35). That is, if $N,M \in \N$ with $N,M \geq 1$, $A \in \mathcal{L}(\R^N)$ and $B \in \mathcal{L}(\R^M;\R^N)$, then the linear ordinary differential system $Y' = AY + Bu$ is controllable at time $T > 0$ if and only if
\begin{equation} \label{kalman}
\rank [A \mid B] = \rank [A^{N-1} B, A^{N-2} B, \cdots, B] = N, 
\end{equation}
where $[A^{N-1}B, A^{N-2}B, \cdots, B] \in \mathcal{L}(\R^{MN};\R^N)$. 

Recently, Liard and Lissy \cite{ll2017} gave a general Kalman condition for the internal controllability of $N$ coupled $d$-dimensional wave equations. 

In the framework of parabolic coupled equations, \cite{abgd2011} gives a general Kalman rank condition for the null boundary controllability of $N$ coupled one-dimensional parabolic equations. The aim of this research is to establish general results, as in \cite{abgd2011}, in the case of one-dimensional coupled wave equations.

To state our results, we recall that the operator $-\partial_{xx}$ on $(0,\pi)$ with homogeneous Dirichlet boundary conditions admits a sequence of eigenvalues $\{\mu_k = k^2\}_{k=1}^\infty$ and eigenfunctions $\{\sin kx\}_{k=1}^\infty$. We note that this family of eigenfunctions is a Hilbert basis of $L^2 (0,\pi)$. 

Our main result is the following:

\begin{thm} \label{thm1}
Suppose that $A$ has $N$ distinct eigenvalues $\lambda_1, \ldots, \lambda_N$. Suppose that the following conditions hold:
\begin{enumerate}[(i)]
\item $[A|b]$ satsifies the Kalman rank condition,
\item \[ \mu_k - \mu_l \neq \lambda_i - \lambda_j, \quad \forall k,l \in \N, \forall 1 \leq  i,j \leq N \text{ with $k \neq l$ and $i \neq j$}, \]
\item $T \geq 2N\pi$.
\end{enumerate}
Then System (1)--(3) is exactly controllable in $\mathcal{H} = H^{N-1} (0,\pi;\R^N) \times H^{N-2} (0,\pi;\R^N)$.

Furthermore, if any of (i), (ii), or (ii) is not satisfied, then System (1)--(3) is not approximately controllable. In particular, if (i) or (iii) does not hold, then the codimension of the reachable set of System (1)-(3) in $L^2 (0,\pi;\R^N) \times H^{-1} (0,\pi;\R^N)$ is infinite. On the other hand, if (ii) fails, the sequence $\{k^2 + \lambda_l\}$, $k \in \N$, $l = 1, \ldots, N$, only contains a finite number of multiple points. So the codimension of the reachable set is finite.
\end{thm}

\begin{rem}With respect to Theorem \ref{thm1}, we have the following remarks.
\begin{itemize}
\item Conditions (i) and (ii) are also necessary conditions that appear in \cite{abgd2011} for the null controllability of $N$ coupled one-dimensional parabolic equations. The hyperbolicity of the equations in our case requires a minimal control time.
\item In general, the reachable space $\mathcal{H}$ is not optimal. In some particular situations it is possible to give an optimal description of the space. Examples include the cases when $N = 2$ or the coupling matrix is cascade, i.e., when $A$ is triangle inferior, or when $A$ is given in canonical form.
\end{itemize}
\end{rem}


\section{The Fourier Method and Existence of Solutions}
In this section, we introduce the Fourier Method. On the assumptions of Theorem \ref{thm1}, we denote $\varphi_1, \ldots, \varphi_N$ to be the family of eigenvectors of $A$ with corresponding eigenvalues $\lambda_1, \ldots, \lambda_N$. We denote by $\langle \cdot, \cdot \rangle$ the inner product in $\R^N$ and so $A^*$ has eigenvalues $\overline{\lambda_i}$ and eigenvectors $\psi_i$ with
\[ \langle \varphi_i, \psi_j \rangle = \delta_{ij}. \]
Let us define $\Phi_{nj}(x) = \sin (nx) \varphi_j$. Then $\{\Phi_{nj}(x)\}$, $n \in \N$, $j = 1, \ldots, N$, is a Riesz basis in $L^2 (0,\pi;\R^N)$ with biorthogonal family $\{\Psi_{nj}(x)\}$ where
\[ \Psi_{nj}(x) = \dfrac{2}{\pi} \sin (nx) \psi_j. \]

We then represent the solution $u$ of System (1) in the form of the series
\begin{equation} \label{fouriersolution} u(x,t) = \sum_{n,j} a_{nj} (t) \Phi_{nj} (x) \end{equation}
and set
\begin{equation} \label{vxt} v(x,t) = g(t) \Psi_{kl} (x), \end{equation}
where $g(t)$ is a smooth function, i.e., $g \in C_0^2 (0,T)$. Below are standard routine manipulations to solve for the coefficients $a_{nj} (t)$:
\begin{align*}
0 &= \int_0^T \int_0^\pi \langle u_{tt} - u_{xx} + Au, v \rangle dx \: dt \\
&= \int_0^T \int_0^\pi \langle u, v_{tt} - v_{xx} + A^* v \rangle dx \: dt + \int_0^\pi \left[ \langle u_t, v\rangle - \langle u, v_t\rangle \right]_{t = 0}^T dx \\
&\quad - \int_0^T \left[ \langle u_x, v \rangle - \langle u, v_x \rangle \right]_{x = 0}^\pi dt\\
&= \int_0^T  \int_0^\pi \langle u, \ddot{g} \Psi_{kl} + k^2 g \Psi_{kl} + \overline{\lambda_l} g \Psi_{kl} \rangle dx \: dt \\
&\quad - \dfrac{2}{\pi} \int_0^T k \langle b, \psi_l \rangle f(t) g(t) \: dt\\
&= \int_0^T a_{kl} [\ddot{g} + (k^2 + \overline{\lambda_l})g] \: dt - \dfrac{2k}{\pi} \int_0^T \langle b, \psi_l \rangle f(t) g(t) \: dt\\
&= \int_0^T [\ddot{a}_{kl} + (k^2 + \overline{\lambda_l})a_{kl}] g \: dt - \dfrac{2k}{\pi} \langle b, \psi_l \rangle \int_0^T f(t) g(t) \: dt.
\end{align*}
Thus we obtain the equations
\begin{equation} \label{distinctdiffeq}
\ddot{a}_{kl} + (k^2 + \overline{\lambda_l})a_{kl} = \dfrac{2k}{\pi} \langle b, \psi_l \rangle f(t)
\end{equation} 
with zero initial conditions that follow from (3), i.e.
\begin{equation} \label{distinctdiffeqic}
a_{kl} (0) = 0 = \dot{a}_{kl}(0).
\end{equation}
We denote $k^2 + \overline{\lambda_l}$ by $\omega_{kl}^2$ and $\langle b, \psi_l \rangle$ by $\beta_l$. In the formulas below we assume that $\omega_{kl}^2 > 0$. In fact, if $\omega_{kl}^2 < 0$ or if $\omega_{kl}$ is not real, we need to replace trigonometric functions by hyperbolic ones (see e.g. \cite{ai1995} Section 3.2). In the case where $\omega_{kl} = 0$, we will set $\frac{\sin (\omega_{kl}t)}{\omega_{kl}} = t$ (see e.g. \cite{ai1995} Sec. III.1).

The solution of (\ref{distinctdiffeq})--(\ref{distinctdiffeqic}) is given by the formula
\begin{equation} \label{distinctakl}
a_{kl} (t) = \dfrac{2k}{\pi} \beta_l \int_0^t f(\tau) \dfrac{\sin \omega_{kl} (t-\tau)}{\omega_{kl}} \: d\tau.
\end{equation}
By differentiating we obtain
\begin{equation} \label{distinctakldot}
\dot{a}_{kl} (t) = \dfrac{2k}{\pi} \beta_l \int_0^t f (\tau) \cos \omega_{kl} (t-\tau) \: d\tau. 
\end{equation}
We now introduce the coefficients
\begin{equation} \label{distinctckl}
c_{kl} (t) = i \omega_{kl} a_{kl} (t) + \dot{a}_{kl} (t).
\end{equation}
We define $\omega_{-kl} = -\omega_{kl}$, $a_{-kl} = a_{kl}$, and $\dot{a}_{-kl} = \dot{a}_{kl}$ for $k \in \K = \{ \pm 1, \pm 2, \ldots\}$, $l \in \{1, \ldots, N\}$, and rewrite (\ref{distinctakl}) and (\ref{distinctakldot}) in the exponential form:
\begin{equation} \label{distinctckl2}
c_{kl} (t) = \dfrac{2k}{\pi} \beta_l \int_0^t f(\tau) e^{i \omega_{kl} (t-\tau)} \: d\tau.
\end{equation}
Taking into account that $\{\Phi_{nj}\}$ forms a Riesz basis in $L^2 (0,\pi;\R^N)$ and 
\[ |\omega_{kl}| + 1 \asymp k, \: k \in \K, \]
we conclude that 
\begin{equation} \label{ckluut}
\sum_{k \in \K} \dfrac{|c_{kl}(t)|^2}{k^2} \asymp \|u(\cdot,t)\|_{L^2(0,\pi,\R^N)}^2 + \|u_t (\cdot,t)\|_{H^{-1}(0,\pi;\R^N)}^2.
\end{equation}

On the other hand, from the explicit form for $\omega_{kl}$, it follows that for any $t> 0$, the family $\{e^{i\omega_{kl}t}\}$ is either the union of a finite number of Riesz sequences if $t < 2\pi N$ or a Riesz sequence in $L^2(0,t)$ if $t \geq 2\pi N$ (see \cite{ai1995} Section II.4). We recall that a Riesz sequence is a Riesz basis in the closure of its linear span (see \cite{ai1995}). Therefore, from (\ref{distinctckl2}) it follows that for every fixed $t > 0$
\begin{equation} \label{cklf}
\sum_{k.l} \dfrac{|c_{kl}(t)|^2}{k^2} \prec \|f\|_{L^2(0,t)}^2.
\end{equation}
Recall that (\ref{ckluut}) and (\ref{cklf}) refer, respectively, to two-sided and one-sided inequalities with constants independent of the sequences $(c_{kl})$, $(k)$, and of the function $f$. 

Additionally, it is not difficult (see \cite{ai1995} Sec.III.1) to check that
\[ \sum_{k,l} \dfrac{|c_{kl}(t+h) - c_{kl}(t)|^2}{k^2} \to 0, \quad h \to 0. \]

We combine our results in the following theorem.

\begin{thm} \label{thmdistinct}
For any $f \in L^2 (0,T)$, there exists a unique generalized solution $u^f$ of the IBVP (1)--(3) such that
\[ (u^f, u_t^f) \in C([0,T];L^2(0,\pi,\R^N) \times H^{-1} (0,\pi;\R^N)) =: \mathcal{V} \]
and
\[ \|(u^f, u_t^f)\|_\mathcal{V} \prec \|f\|_{L^2(0,T)}. \]
\end{thm}


\section{Controllability Results}

In this section we will prove Theorem \ref{thm1}. We assume that Conditions (i), (ii), and (iii) are satisfied. By Proposition 3.1 in \cite{abgd2011}, Condition (i) implies that $\beta_l \neq 0$ for all $l = 1, \ldots, N$. We then define $\gamma_{kl}$ to be
\begin{equation} \label{distinctgammakl} \gamma_{kl} := c_{kl}(T) \left( \dfrac{2k}{\pi} \beta_l e^{i \omega_{kl} T} \right)^{-1} \end{equation}
and rewrite (\ref{distinctckl2}) for $t = T$ in the form
\[ \gamma_{kl} = (f, e_{kl})_{L^2 (0,T)}, \]
where $e_{kl}(t) = e^{i \omega_{kl} t}$. We note that
\[ \sum_{k,l} |\gamma_{kl}|^2 \asymp \sum_{k,l} \dfrac{|c_{kl}(T)|^2}{k^2}. \]

We note that for $k$ fixed, the points $\omega_{kl}$ for $l = 1, \ldots, N$ are asymptotically close, i.e., these $N$ points lie inside an interval whose length tends to zero as $k$ tends to infinity. Therefore, the family $\{e_{kl}\}$ is not a Riesz basis in $L^2 (0,T)$ for any $T$. We therefore need to use the so-called exponential divided differences (EDD).

EDD were introduced in \cite{ai2001} and \cite{am2001} for families of exponentials whose exponents are close, that is, the difference between exponents tends to zero. Under precise assumptions, the family of EDD forms a Riesz basis in $L^2 (0,T)$. For each fixed $k$, we define
\[ \tilde{e}_{k1} := [\omega_{k1}] = e^{i\omega_{k1}t}, \]
and for $2 \leq l \leq N$
\[ \tilde{e}_{kl} := [\omega_{k1}, \omega_{k2}, \ldots, \omega_{kl}] = \sum_{j=1}^l \dfrac{e^{i \omega_{kj}t}}{\prod_{r \neq j} (\omega_{kj} - \omega_{kr})}. \]
Under Condition (ii) of our theorem, we are able to use this formula for divided differences as opposed to the formula for generalized divided differences (see e.g. \cite{am2001}). 

From asymptotics theory and the explicit formula for $\omega_{kl}$, it follows that the generating function of the family of EDD $\{\tilde{e}_{kl}\}$ is a sine-type function (see \cite{ai1995, ai2001, am2001}). Hence, the family of EDD $\{\tilde{e}_{kl}\}$ forms a Riesz sequence in $L^2 (0,T)$ for $T \geq 2\pi N$. We then define
\[ \tilde{\gamma}_{kl} = (f, \tilde{e}_{kl})_{L^2 (0,T)}. \]
Since $\{ \tilde{e}_{kl}\}$ is a Riesz sequence, $\{ (\tilde{\gamma}_{kl}) \mid f \in L^2 (0,T)\} = \ell^2$, i.e. any sequence from $\ell_2$ can be obtained by a function $f \in L^2 (0,T)$ and the family $\{ \tilde{e}_{kl}\}$. We note that $| \omega_{kj} - \omega_{ki}| \asymp k^{-1}$, where $1 \leq i,j \leq N$. In particular, this implies that $|\tilde{\gamma}_{kl}| \prec k^{N-1} |\gamma_{kl}|$. Recalling Equations (\ref{distinctckl}) and (\ref{distinctgammakl}), we obtain
\begin{equation} \label{tildegkl}
\{ (\gamma_{kl}) \mid f \in L^2 (0,T)\} \supseteq \ell_{N-1}^2 
\end{equation}
where
\[ \ell_{N-1}^2 = \left\{ (a_{kl}) \mid \sum_{k.l} |k^{N-1} a_{kl}|^2 < \infty \right\}. \]
Since $\{\Phi_{kl}\}$ forms a Riesz basis in $L^2 (0,\pi;\R^N)$, from (\ref{distinctckl}), (\ref{distinctgammakl}), (\ref{tildegkl}), $(u(\cdot,t),u_t (\cdot,t)) \in H^{N-1}(0,\pi;\R^N) \times H^{N-2} (0,\pi,\R^N)$ and we have proved Theorem \ref{thm1}.

We will now prove the negatives results in Theorem \ref{thm1}. We first assume that (i) and (iii) hold, but (ii) does not hold. Observe that this may only happen for a finite number of indices (see \cite{abgd2011}). So we have
\[ k_d^2 - l_d^2 = \lambda_{i_d} - \lambda_{j_d}, \quad 1 \leq d \leq m. \]
In this situation, the family given in (\ref{distinctckl2}), $\{e_{kl}\}$, is clearly linearly dependent since some function (or functions) is repeated twice in the family. Thus, according to Theorems I.2.1e and III.3.10e in \cite{ai1995}, System (1) is not approximately controllable for any $T > 0$. 

Let us now suppose that (i) does not hold. This case is proved directly and is related to properties of exponential families (see \cite{ai1995} Sections I.1 and III.1). 

If condition (iii) is not met, i.e. $T < 2\pi N$, then from \cite{ai2001} and \cite{am2001}, it follows that the family of EDD $\{\tilde{e}_{kl}\}$ is not a Riesz basis in $L^2 (0,T)$. In particular, we can split $\{\tilde{e}_{kl}\}$ into two subfamilies $\mathcal{E}_0$ and $\mathcal{E}_1$ such that $\mathcal{E}_0$ is a Riesz sequence in $L^2 (0,T)$ and $\mathcal{E}_1$ has infinite cardinality. This implies that $\{\tilde{e}_{kl}\}$ is not linearly independent and hence the reachable set has infinite codimension.

Thus we have proved the negative part of Theorem \ref{thm1}, and the proof is complete.


\section{A Particular Case: $N = 2$}

In the previous sections, we proved exact controllability with respect to a more regular space than the space of regularity for the system. This is typical of hybrid systems where clusters of close spectral points appear. However, in the case where $N = 2$, we are able to prove the sharp controllability result, i.e., to prove exact controllability in the space of sharp regularity of the system. To do this, we develop a new method based on the construction of a basis in a so-called asymmetric space. This method was proposed in \cite{ae2018} when investigating the controllability of another hybrid system of hyperbolic type -- the string with point masses. In the present paper, we extend this method to the vector case.

We consider System (1)-(3) with $N = 2$ and
\begin{equation} \label{N2mat} b = \begin{pmatrix} 1 \\ 0 \end{pmatrix}, \quad A = \begin{pmatrix} a_{11} & a_{12} \\ a_{21} & a_{22} \end{pmatrix}.\end{equation}
In other words, the boundary control acts only on the first equation and the second equation is controlled through its connection with the first. From now on, we will refer to this system as $\mathcal{S}_2$. The first question we ask is about the sharp regularity space. We claim that
\[ u_1 (\cdot, t) \in L^2 (0,\pi), \quad u_2 (\cdot, t) \in H^1_0 (0,\pi). \]

From Theorem \ref{thmdistinct}, $(u_1 (\cdot, t), u_2 (\cdot,t)) \in L^2 (0,\pi)^2$. From the structure of the system, $u_2$ is a solution to a wave equation with zero Dirichlet boundary conditions and only depends on $u_1$. In particular, $u_2$ can be solved as a system of linear nonhomogeneous ordinary differential equations. Using standard methods to solve this system yields that $u_2 \in H^1_0 (0,\pi)$. 

The main result of this section is
\begin{thm} \label{thmN2}
Under conditions similar to those of Theorem \ref{thm1}, that is, assume that $A$ has two distinct eigenvalues $\lambda_1$, $\lambda_2$ and $b$ given by (\ref{N2mat}) with $a_{21} \neq 0$ (so the Kalman rank condition for $[A|b]$ is fulfilled), that
\[ \mu_k - \mu_l \neq \lambda_1 - \lambda_2, \quad \forall k,l \in \N, \text{ with $k \neq l$}, \]
and that $T \geq 4\pi$, then the reachable set of System $\mathcal{S}_2$, $\{(u^f (\cdot, T), u_t^f (\cdot,T)) \mid f \in L^2 (0,T)\}$ is equal to $\mathcal{H}_1$ where
\[ \mathcal{H}_1 := \begin{pmatrix} L^2 (0,\pi) \\ H^1_0 (0,\pi) \end{pmatrix} \times \begin{pmatrix} H^{-1} (0,\pi) \\ L^2 (0,T) \end{pmatrix} \]
for $T \geq 4\pi$. 

If $T < 4\pi$, then the reachable set has infinite codimension in $\mathcal{H}_1$. 
\end{thm}

We will prove this theorem by considering the two possible cases, i.e., whether the matrix $A$ has two distinct eigenvalues or a repeated eigenvalue. 

\begin{proof}
We now return to the representation in (\ref{fouriersolution}):
\begin{equation} \label{fouriersolutionN2}
u(x,T) = \sum_{n,j} a_{nj} (T) \Phi_{nj} (x).
\end{equation}

Taking into account that for $N = 2$, we use EDD of order one, i.e., 
\[ \tilde{a}_{n1} = a_{n1}, \quad \tilde{a}_{n2} = \dfrac{a_{n2}-a_{n1}}{\omega_{n2} -\omega_{n1}}, \]
where we supress the argument $T$. We can rewrite (\ref{fouriersolutionN2}) in the form
\begin{equation} \label{fouriersolutionN22}
u(x,T) = \sum_{n,j} \tilde{a}_{nj} \tilde{\Phi}_{nj} (x).
\end{equation}
It is easy to verify that
\begin{align}
\tilde{\Phi}_{n1} (x) &= \Phi_{n1} (x) + \Phi_{n2} (x) = \sin (nx) (\varphi_1 + \varphi_2), \label{tildephi1}\\
\tilde{\Phi}_{n2} (x) &= \Phi_{n2} (x) (\omega_{n2} - \omega_{n1}) = \sin (nx) \varphi_2 (\omega_{n2} - \omega_{n1}). \label{tildephi2}
\end{align}
We note that $| \omega_{n2} - \omega_{n1}| \asymp n^{-1}$. We present the following lemma.
\begin{lem}
Eigenvectors $\varphi_1$ and $\varphi_2$ can be chosen such that
\begin{equation} \label{phi12} \varphi_1 + \varphi_2 = \begin{pmatrix} \alpha \\ 0 \end{pmatrix}. \end{equation}

\begin{proof}
In particular, we claim that the second component of $\varphi_1$ and $\varphi_2$ are nonzero. If this is true, then by appropriate scaling, we can obtain eigenvectors $\varphi_1$ and $\varphi_2$ whose second components add to zero. Suppose on the contrary that $\varphi_1$ has a zero second component. By scaling, we can assume that 
\[ \varphi_1 = \begin{pmatrix} 1 \\ 0 \end{pmatrix}. \]
By the orthogonality of $\psi_1, \psi_2$, this implies that $\psi_2$ has the form
\[ \psi_2 = \begin{pmatrix} 0 \\ x \end{pmatrix}, \]
for some nonzero $x$. However, this is a contradiction to the Kalman rank condition as 
\[ \left\langle \begin{pmatrix} 1 \\ 0 \end{pmatrix}, \psi_2 \right\rangle = 0. \]
Hence, both $\varphi_1$ and $\varphi_2$ have nonzero second components and the lemma is proved.
\end{proof}
\end{lem}
We can now express (\ref{fouriersolutionN22}) as
\[ u(x,T) = \sum_n \sin (nx) \left[ \tilde{a}_{n1} \begin{pmatrix} \alpha \\ 0 \end{pmatrix} + \tilde{a}_{n2} \begin{pmatrix} \beta \\ \gamma \end{pmatrix} (\omega_{n2} - \omega_{n1}) \right]. \]
We note that it is clear that $\gamma \neq 0$. 

We recall that $(\tilde{a}_{n1})$ and $(\tilde{a}_{n2})$ may be arbitrary $\ell^2$ sequences (when $f$ runs over $L^2 (0,T)$). Taking into account that $\{ \sin (nx)\}$ is an orthogonal basis in $L^2 (0,\pi)$, we begin by choosing the second component of $u(x,T)$ to be any target function from $H^1_0 (0,\pi)$, and thereby choosing $\tilde{a}_{n2}$ (recalling that $|\omega_{n2} - \omega_{n1}| \asymp n^{-1}$). 

After choosing $\tilde{a}_{n2}$, we can then choose $\tilde{a}_{n1}$ so that the first component of $u(x,T)$ will coincide with any prescribed function from $L^2 (0,\pi)$. We can treat $u_t (x,T)$ in a similar fashion. This is due to the relation of sine and cosine and their appearance in $u (x,T)$ and $u_t(x,T)$. It is this relation that allows us to obtain controllability in time $T \geq 4\pi$. Thus, one of the cases for the positive part of Theorem \ref{thmN2} is proved. We note that the negative part of the theorem can be proved similar to Theorem \ref{thm1}. 
\end{proof}

As a result of this, we have the following corollary.
\begin{cor}
The family $\{\tilde{\Phi}_{nj}\}$ constructed in (\ref{tildephi1})--(\ref{phi12}) forms a Riesz basis in the asymmetric space $L^2 (0,\pi) \times H^1 (0,\pi)$.
\end{cor}
\begin{proof}
We have proved that every function from $L^2 (0,\pi) \times H^1 (0,\pi)$ can be represented in the form of a series with respect to the family $\{ \tilde{\Phi}_{nj}\}$ with $\ell^2$ coefficients. Uniqueness of the representation follows from the basis property of $\{\sin (nx)\}$ and linear independence of the eigenvectors $\varphi_1$ and $\varphi_2$. Finally, it is clear that
\[ \| u_1 (\cdot, T)\|_{L^2 (0,\pi)}^2 + \|u_2 (\cdot, T)\|_{H^1 (0,\pi)}^2 \asymp \sum_{n,j} | a_{nj}|^2. \]
\end{proof}

As a remark, the latter sum is equivalent to $\|f\|^2$ where $f$ is the corresponding control to $u(\cdot, T)$ with the minimal norm. This control belongs to the closure of the linear span of $\{e^{i \omega_{nj} t}\}$ in $L^2 (0,T)$.


\section{Open Problems and Further Results}

When the coupling matrix $A$ is in lower triangular form, it is not difficult to generalize the results for coupled hyperbolic equations. That is, it is possible to prove exact controllability under the same assumptions as Theorem \ref{thm1} in the space $\mathcal{H} =\mathcal{H}^0 \times \cdots \times \mathcal{H}^{N-1}$ where $\mathcal{H}^N = H^N (0,\pi) \times H^{N-1} (0,\pi)$. On the other hand, given an arbitrary matrix $A$, if the Kalman rank condition holds, we can obtain a canonical version of the original system and obtain similar results for this \emph{transformed system}. The problem that arises is going back to the original system combines the different components in $\mathcal{H}$ and an optimal description of the controllability space is no longer possible.

While we have proved controllability for this system, we assume that the coupling matrix $A$ has $N$ distinct eigenvalues. It remains to be proved that the system is controllable for a generic matrix $A$, assuming that the Kalman rank condition is satisfied.

It remains an open problem to treat the boundary controllability of $N$ coupled wave equations in $\R^d$. The methods in this paper are not of use in the general situation or when the matrix $A$ depends on $(x,t)$.

When finishing the writing of this paper, \cite{bat2017} was published. In it, the case of two coupled one-dimensional wave equations with first order coupling and a specific coupling matrix $A = A(x)$ was treated.

\section{Acknowledgments}
A significant part of this research was made when  S.~Avdonin visited UNAM supported by
PREI, UNAM, Mexico. He is very grateful to the Department of Mathematics for its hospitality. S.~Avdonin  was also supported in part by NSF grant DMS
1411564 and by the Ministry of Education
and Science of Republic of Kazakhstan under the grant No.
AP05136197. L. de Teresa was supported in part  by PAPIIT-IN102116, UNAM, Mexico.

\bibliographystyle{ieeetr}

\end{document}